\documentclass[reqno]{amsart}

\usepackage{physics}
\usepackage{ascmac}
\usepackage{algorithm}
\usepackage{algorithmic}
\usepackage{mathtools}
\usepackage{bbm}
\mathtoolsset{showonlyrefs=true}
\usepackage{color}
\usepackage[hidelinks]{hyperref}

\DeclarePairedDelimiter\paren{\lparen}{\rparen}
\DeclarePairedDelimiterX{\inpr}[2]{\langle}{\rangle}{{#1},{#2}}
\DeclarePairedDelimiterX{\setI}[2]{\{}{\}}{\,{#1}\ \delimsize| \ {#2}\,}

\newcommand{\RR}{\mathbb{R}}
\newcommand{\CC}{\mathbb{C}}

\newcommand{\x}[1]{x^{\paren*{#1}}}

\newcommand{\re}{\mathrm{e}}
\newcommand{\ri}{\mathrm{i}}
\newcommand{\Order}[1]{\mathrm{O} \paren*{#1}}

\theoremstyle{plain}
\newtheorem{thm}{Theorem}
\newtheorem{prop}[thm]{Proposition}
\theoremstyle{definition}
\newtheorem{dfn}[thm]{Definition}
\theoremstyle{remark}
\newtheorem*{rem}{Remark}

\begin{document}

\author[K.~Ushiyama]{Kansei Ushiyama}
\address[K.~Ushiyama]{Graduate School of Information Science and Technology, the University of Tokyo, 7-3-1, Hongo, Bunkyo-ku, Tokyo 113-8656, Japan}
\email{ushiyama-kansei074@g.ecc.u-tokyo.ac.jp}

\author[S.~Sato]{Shun Sato}
\address[S.~Sato]{Graduate School of Information Science and Technology, the University of Tokyo, 7-3-1, Hongo, Bunkyo-ku, Tokyo 113-8656, Japan}

\author[T.~Matsuo]{Takayasu Matsuo}
\address[T.~Matsuo]{Graduate School of Information Science and Technology, the University of Tokyo, 7-3-1, Hongo, Bunkyo-ku, Tokyo 113-8656, Japan}

\keywords{optimization, dynamical systems, ordinary differential equations, numerical analysis}
%
%
\title[Essential convergence rate of ODEs in optimization]{Essential convergence rate of ordinary differential equations appearing in optimization}
%
%
\begin{abstract}
Some continuous optimization methods can be connected to ordinary differential equations (ODEs) by taking continuous limits, and their convergence rates can be explained by the ODEs. However, since such ODEs can achieve any convergence rate by time scaling, the correspondence is not as straightforward as usually expected, and deriving new methods through ODEs is not quite direct. In this letter, we pay attention to stability restriction in discretizing ODEs and
show that acceleration by time scaling basically implies deceleration in discretization; they balance out so that we can define an attainable unique convergence rate which we call ``essential convergence rate''.
\end{abstract}
\maketitle
\section{Introduction}
Continuous optimization and ordinary differential equations (ODEs) are closely related.
For example, given an unconstrained optimization problem 
\begin{equation}
       \min_{x \in \RR^d} f(x),
\end{equation}
the steepest descent method, the simplest method for it,
\begin{equation}
       \x{k+1} = \x{k} - h_{k+1}\nabla f(\x{k}),
\end{equation}
can be interpreted as the explicit Euler method for the gradient flow $\dot{x} = -\nabla f(x)$.
Here, $h_{k+1}$ can be read as the time step size in the context of numerical methods.
The convergence rate of the steepest descent method for convex and $L$-smooth functions is $f(\x{k}) - f^\star = \Order{1/k}$ ($f^\star := \min_{x\in \RR^d} f(x)$), which corresponds to the rate of the gradient flow for convex functions $f(x(t)) - f^\star = \Order{1/t}$. 

This point of view has been becoming more important, especially after the derivation of the second-order ODE
\begin{equation}
       \ddot{x} + \frac{3}{t} \dot{x} + \nabla f(x) = 0 \label{suODE}
\end{equation}
as a continuous limit of Nesterov's accelerated gradient descent method (NAG) by Su--Boyd--Cand\`{e}s~\cite{Su}.
NAG achieves the optimal convergence rate $\Order{1/k^2}$ for $L$-smooth convex functions.
Since the ODE~\eqref{suODE} also achieves the rate $\Order{1/t^2}$, the discrete and continuous convergence rates nicely matches.
Similarly, for various known optimization methods, their continuous limit ODEs have been derived and it has been shown that the convergence rates are consistent~\cite{Wilson}.

Interpreting optimization methods as a pair of continuous dynamical systems and their discretizations (numerical methods) provides useful insights.
For example, the agreement between the discrete and continuous convergence rates suggests that the essence of the convergence mechanism lies in the underlying dynamical systems where our intuition works, and it also suggests that numerical methods that bridge continuous and discrete can preserve the rate.
Another advantage of this view is that we can prove the convergence rate in the following two steps: analyzing the convergence rate of the ODE, and evaluating its discretization.
This in turn is expected to help us construct new optimization methods.

However, this correspondence is not quite complete in the following sense.
For discrete optimization methods, the lower bound of the convergence rate is known for various objective function classes (cf.~\cite{Nesterov}).
On the other hand, we cannot naively consider the lower bound of the convergence rate for ODEs, because the rate can be arbitrarily changed by nonlinear time rescaling as pointed out in~\cite{Wibisono}.
Moreover, as even the simple gradient flow can achieve arbitrary rates, it is difficult to determine which ODE is best for the optimization method.

In this letter, in order to fill the missing gap we show that the indefiniteness caused by time rescaling can be eliminated by defining {\em essential convergence rate} in continuous systems with the help of the concept of stability in numerical analysis.

\section{Preliminary: stability}
When implementing an optimization method, we have to choose step sizes small enough to avoid overflow.
Once we regard optimization methods as numerical methods for ODEs, we can see that possible step size is determined by the stability of the numerical method through linear stability analysis.
\begin{dfn}\label{stability}
       (cf.~\cite{Hairer}) Let $R(h\lambda)$ be the value obtained by applying the numerical method to Dahlquist's test equation $\dot{y} = \lambda y, y(0) = 1$. $R$ is called the stability function and the set $\{z \in \CC \mid |R(z)| \le 1\}$ is called the {\em stability domain}.
\end{dfn}
\begin{rem}
       Although Definition~\ref{stability} gives the stability condition of the numerical solution only for the linear scalar ODE, the stability for the multi-dimensional nonlinear ODE $\dot{y} = g(y)$ can be similarly handled by identifying $\lambda$ as each eigenvalue of the Jacobian of $g$. 
\end{rem}

The stability domains of explicit numerical methods are basically bounded. For example, the stability function of any explicit Runge--Kutta method is polynomial, and therefore the stability domain should be bounded.
In the following sections, since computationally expensive methods are not suitable for general optimization methods, explicit methods are assumed to be used.

\section{Essential convergence rate}\label{sec3}
In this section, we consider the essential convergence rate.
In order to allow high-order ODEs such as~\eqref{suODE}, we consider the following $d'$-dimensional ($d' \ge d$) first-order non-autonomous system $\dot{y} = g(y,t)$,
where $y_i = x_i \, (i = 1,\dots,d)$, and if $d' > d$ the remaining $y_{d+1},\dots,y_{d'}$ are necessary auxiliary variables. Accordingly we extend the objective function $f$ and the optimal solution to $d'$-dimension by $\tilde{f} (y) = f(y_1,\dots,y_d)$. 
We define $\mathcal{G}$ as the collection of the differentiable vector fields $g : \RR^{d'} \times \RR_{\ge 0} \to \RR^{d'}$ such that 
$\lim_{t \to \infty} \tilde{f}(y(t)) = f^\star$ from any initial point.
Below we abuse the notation and simply denote the objective function by $f$.

We consider time rescaling for the solution $y(t)$ of the ODE $\dot{y} = g(y,t)$.
Time rescaling is change of variables $t = \alpha(\tau)$ where $\alpha$ is a monotonically increasing differentiable function $\alpha : \RR_{\ge 0} \to \RR_{\ge 0}$ with $\alpha(0) = 0, \lim_{t\to\infty}\alpha(t)=\infty$.
By applying time rescaling $t=\alpha(\tau)$ to $y(t)$, we obtain a time rescaled solution $\tilde{y}(\tau) := y(\alpha(\tau))$. 
Then $\tilde{y}$ follows the new ODE 
\begin{equation}
       \dv{\tilde{y}}{\tau} = \dv{\tau}y(\alpha(\tau)) = \dot{\alpha}(\tau)g(\tilde{y}(\tau),\alpha(\tau)). \label{ode4:trans}
\end{equation}
In this way we obtain two different ODEs that share the same trajectory.

\begin{dfn}\label{dfnsim}
       For $g_1,g_2 \in \mathcal{G}$, we consider two ODEs $\dot{y}_1 = g_1(y_1,t)$ and $\dot{y}_2 = g_2(y_2,t)$. If there exists a time rescaling function $\alpha$ such that the solutions $y_1$ and $y_2$ of these ODEs w.r.t. an arbitrary same initial point satisfy $y_2(t) = y_1(\alpha(t))$, we write $g_1 \sim g_2$.
       The symbol $\sim$ defines an equivalence relation in $\mathcal{G}$, and we denote the equivalence class for $g \in \mathcal{G}$ by $[g]$.
\end{dfn}
In the above situation, $g_1$ and $g_2$ satisfy $g_2(y_2(t),t) = \dot{\alpha}(t)g_1(y_1(\alpha(t)),\alpha(t)) = \dot{\alpha}(t)g_1(y_2(t),\alpha(t))$.

Now let us consider applying numerical methods to these ODEs.
As said before, we only consider numerical methods with bounded stability domains. This implies that in~\eqref{ode4:trans} all the eigenvalues of $h_k\dot{\alpha}(t)(\partial g/\partial y)$  should stay in the domains for every $\alpha(t)$ chosen; thus if, for example, $\dot{\alpha}(t)\rho(\partial g/\partial y)\to\infty$ as $t\to\infty$ ($\rho$ is the spectral radius), we are forced to take decreasing time step widths $h_k\to 0$, and the overall efficiency may not improve. We also notice that among various $\alpha(t)$'s if one realizes $\dot{\alpha}(t)\rho(\partial g/\partial y)=\Theta(1)$, that would be a good choice since it should allow a simple fixed time-stepping implementation. Under these observations, we introduce the following definition.
\begin{dfn}\label{dfnrate}
       For $[g]$, $g_0 \in [g]$ is said to be a {\em proper representative of} $[g]$ if it satisfies $\rho\left({\partial g_0((y,t))}/{\partial y} \right) = \Theta(1)$. When there is a proper representative in $[g]$, the {\em essential convergence rate of} $[g]$ is defined by $\beta(t)$ such that $f(y(t)) - f^\star = \Theta(\beta(t))$ holds for any solution $y$ of the ODE corresponding to the proper representative.
\end{dfn}

The concept of essential convergence rate might not seem well-defined when there are multiple proper representatives in an equivalence class. The next proposition reveals, however, it is actually valid; i.e., the corresponding rates coincide up to linear rescalings of time.
\begin{prop}
       Let $g\in \mathcal{G}$ be given. If $\beta_1, \beta_2$ are essential convergence rates of $[g]$, there exist $C_1, C_2 \in \RR_{>0}$ such that $\beta_2(t) = \Order{\beta_1(C_1t)}$ and $\beta_1(t) = \Order{\beta_2(C_2t)}$ holds.
\end{prop}

\begin{proof}
       Let $g_1,g_2 \in [g]$ be proper representatives. Then there exists a time rescaling function $\alpha(t)$ such that for any solution $g_2(y,t) = ({{\rm d}\alpha(t)}/{{\rm d}t}) g_1(y,\alpha(t))$ holds. From the definition of proper representative, we see $\dot{\alpha}(t)=\Theta(1)$.
       This implies for sufficiently large $t$ there exists a constant $C_1 \in \RR_{>0}$ and
       \begin{equation}
         C_1 t \le \alpha(t) \label{a}
       \end{equation}
       holds.
       Let $\beta_1(t)$ (respectively, $\beta_2(t)$) be the convergence rate derived from ${{\rm d}y_1}/{{\rm d}t} = g_1(y_1,t)$ (resp. ${{\rm d}y_2}/{{\rm d}t} = g_2(y_2,t)$). It follows from~\eqref{a} and
       \begin{align}
         \Theta( \beta_2(t) ) &= f(y_2(t)) - f^\star \\
         &= f(y_1(\alpha(t))) - f^\star = \Theta( \beta_1(\alpha(t)) )
       \end{align}
       that $\beta_2(t) = \Order{\beta_1(C_1 t)}$. Similarly, $\beta_1(t) = \Order{\beta_2(C_2 t)}$ holds.
\end{proof}

Now we show that with time rescalings the convergence rates intrinsically cannot exceed the essential one in Definition~\ref{dfnrate}, if we take discretization into account as well. 

We start by clarifying our setting for the theorem. 
Let us suppose we are given $g\in\mathcal{G}$, and there is a proper representative $g_0$ in $[g]$. Below we only consider the time rescaling $\alpha$ from $g_0$ with monotonic $\dot{\alpha}$, i.e., the rescaling of the form $\tilde{g}(y,t) = \dot{\alpha}(t)g_0(y,\alpha(t))$, and consider the behavior of this $\alpha(t)$. 
The assumption on $\dot{\alpha}$ is rather a technical condition for the main theorem, but it is satisfied by typical accelerating (or decelerating) time rescalings such as $t^{p}$ ($p > 0$),
$\exp(t)$, and $\log(t)$.
Let us denote time step widths by $h_k$, and the corresponding time grids by $t_k := \sum^{k}_{i = 1} h_i$ ($k=1, 2,\ldots$). We denote the numerical solution by $y^{(k)}$ ($k=0, 1,\ldots$).

\begin{thm}\label{thm4}
      Suppose we employ a numerical method whose stability domain is bounded and static (i.e., it does not change with time). Suppose also that, for any  $\tilde{g}\in[g]$ chosen, we control  time step widths so that all the eigenvalues of $h_k \pdv{\tilde{g}}{y}\,\!(y^{(k-1)},t_{k-1}) \, (k=1,2,\dots)$ lie in the stability domain.
      Then for each element of $[g]$ with monotonic $\dot{\alpha}$, there exists $T \in \RR_{\ge 0}$ and an associated discrete time grid $t_{k_0}\ge T$ such that
       \begin{equation}
              \alpha(t_{k_0+k}) - \alpha(t_{k_0}) = \Order{k}
       \end{equation}
       holds.
\end{thm}

Before going to the proof, we mention the meaning of Theorem~\ref{thm4}. Notice that $t_{k_0+k} - t_{k_0}$ is the (discrete) elapsed time in the time scale of $\tilde{g}$, while $\alpha(t_{k_0+k}) - \alpha(t_{k_0})$ denotes the one in the scale of $g_0$. The claim that the latter being $\Order{k}$ implies that however fast the rate might seem in the ``(hopefully) accelerated'' ODE $\dot{y} = \tilde{g}(y,t)$, if we measure the elapsed time on $g_0$'s time scale during solving $\tilde{g}$ ODE for $k$ steps, it is actually nothing more than the time during $k$-step integration of $g_0$ ODE with a fixed time step width.
Thus the convergence rate of discretized $\tilde{g}$ cannot be faster than the discretized $g_0$ with a fixed time step width.

\begin{rem}
       $T$ and $k_0$ are introduced for a technical reason, to counter some singular ODEs such as~\eqref{suODE}, where $\rho(\partial g_0/\partial y)$ tends to infinity as $t\to 0$. In such a case, we need to cut off a short interval around the origin and construct a theorem in the remaining region. 
       In other normal cases, we can simply take $T=0$ and $k_0=0$.
\end{rem}
\begin{proof}
       Let us first clarify the restriction on the time step widths. Since $g_0$ is a proper representative, there exist $T > 0$ and $c > 0$ such that $\rho({\partial g_0}/{\partial y}) \ge c$ for any $t > T$. Below we only consider this time region which is enough to discuss an asymptotic convergence rate. Let us take a discrete time $t_{k_0}>T$ and fix it throughout this proof. From the assumption on the numerical method, there is a constant $r>0$ coming from the size and shape of the stability domain, and the time step size $h_k$ should satisfy
       $\abs{ h_k \dot{\alpha}(t_{k-1}) \rho ( {\partial g_0}/{\partial y} )} \le r$, 
       i.e.,
       \begin{equation}
              h_k \dot{\alpha}(t_{k-1}) \le \frac{r}{c} \quad (k=k_0+1, k_0+2, \ldots). \label{b}
       \end{equation}

       With this observation, a rough sketch of the proof is immediate:
       \begin{align}
              & \alpha(t_{k_0+k}) - \alpha(t_{k_0})
               = \sum_{i = k_0 + 1}^{k_0+ k} \int_{t_{i-1}}^{t_i} \dot{\alpha}(t) \dd t \\
              & \simeq  \sum_{i = k_0 + 1}^{k_0+k} h_i \dot{\alpha}(t_{i-1})
              \le  \frac{r}{c} k 
              =  \Order{k}. \label{rough}
       \end{align}
      When $\dot{\alpha}$ is weakly monotonically decreasing, $\simeq$ can be replaced with $\le$ by~\eqref{b}, and the proof is complete.
      Thus we just need to prove the other case. 

       When $\dot{\alpha}$ is weakly monotonically increasing, we have instead
       \begin{align}
              \alpha(t_{k_0+k}) - \alpha(t_{k_0})
              = \sum_{i = k_0+1}^{k_0+k} \int_{t_{i-1}}^{t_i} \dot{\alpha}(t)\dd t \\
              \le  \sum_{i =k_0+ 1}^{k_0+k} h_i \dot{\alpha}(t_{i})
              \le  \sum_{i = k_0+1}^{k_0+k} \frac{r}{c} \frac{\dot{\alpha}(t_i)}{\dot{\alpha}(t_{i-1})}. \label{inc}
       \end{align}
       From this, we see that if $E(t_i):=\dot{\alpha}(t_i)/\dot{\alpha}(t_{i-1})$ $(i=k_0+1, k_0+2,\ldots)$ is bounded the $\Order{k}$ claim is obvious. Otherwise $\{E(t_i)\}$ includes an unbounded subsequence; to counter such cases, let us consider the subsequence extracting ``large'' elements:
       \begin{align}
              \{E(t_{i_j})\}_{j=1}^{\infty} = \{ E(t_i)\, | \, E(t_i) > 1 + \varepsilon)\}_{i=k_0+1}^{\infty},
       \end{align}
       where $\varepsilon > 0$ is an arbitrary fixed constant. For convenience, we set $E(t_{i_0}) = 1$ and $t_{i_0} = t_{k_0}$.
       Let us here also introduce $J_k$ as the largest index $j$ such that $i_{j} \le k_0 + k$.
       With these notation, if we admit an estimate:
       \begin{equation}
              \frac{E(t_{i_{j}})}{i_{j+1} - i_j} 
              = \Order{1}, \label{E:estimate}
       \end{equation}
       the proof would complete as follows.
       From~\eqref{inc} we see
       \begin{align}
              &\alpha(t_k) - \alpha(t_0)\\
              \le& \Order{k} + \sum_{j=1}^{J_k} \frac{r}c E(t_{i_j}) \\
              \le& \frac{r}c \sum_{j=1}^{J_k} (i_{j+1} - i_j) \left(\frac{E_{t_{i_j}}}{i_{j+1} - i_j}\right) + \Order{k} \\
              \le& \frac{r}c \sum_{j=1}^{J_k} (i_{j+1} - i_j) \Order{1} + \Order{k} 
              = \Order{k},
       \end{align}
       where in the first inequality we split the sum in the last term of~\eqref{inc} in the ``large'' elements and the rest, and the $\Order{k}$ term comes from the latter. 

       Now let us show~\eqref{E:estimate}.
       Since $\dot{\alpha}$ is weakly monotonically increasing, and since $E(t_i) \ge 1$ for $i \ge k_0$ and $E(t_{i_j}) > 1 + \varepsilon$ for $j \ge 1$,
       \begin{align}
              &\sum_{i=k_0+1}^{k_0+k} \frac{1}{\dot{\alpha}(t_{i-1})}\\
              \le& \sum_{j = 0}^{J_k} \frac{1}{\dot{\alpha}(t_{i_j})} (i_{j+1} - i_j)\\
              =& \sum_{j = 0}^{J_k} \left[ \frac{1}{\dot{\alpha}(t_{k_0})} \paren*{\prod_{l=k_0+1}^{i_j}\frac{1}{E(t_l)}} (i_{j+1} - i_j) \right]\\
              \le& \sum_{j = 0}^{J_k} \left[ \frac{1}{\dot{\alpha}(t_{k_0})} \paren*{\prod_{j'=1}^{j}\frac{1}{E(t_{i_{j'}})}} (i_{j+1} - i_j) \right]\\
              <& \sum_{j = 0}^{J_k} \frac{1}{\dot{\alpha}(t_{k_0})} \frac{1}{(1+\varepsilon)^{j-1}}\frac{1}{E(t_{i_{j}})} (i_{j+1} - i_j). \label{J}
       \end{align}
       Here if we take the limit of $k\to\infty$, the most left hand side should tend to $\infty$, since
       \begin{equation}
              \lim_{k \to \infty} (t_{k_0+k} - t_{k_0}) 
              = \lim_{k \to \infty}\sum_{i=k_0+1}^{k_0+k} h_i 
              \le \lim_{k \to \infty} \frac{r}{c} \sum_{i=k_0+1}^{k_0+k} \frac{1}{\dot{\alpha}(t_{i-1})}.
       \end{equation}
       If we demand $t_k\to\infty$ as $k\to\infty$ (which is necessary for a numerical method to make sense; recall $\alpha$ is a time rescaling function), $\sum 1/\dot{\alpha}$ should be so as well.
       Thus,
       \begin{equation}
              \sum_{j = 1}^{\infty} \frac{1}{(1+\varepsilon)^{j-1}}\frac{1}{E(t_{i_{j}})} (i_{j+1} - i_j) = \infty.
       \end{equation}
       Since $\sum_{j=1}^{\infty} 1/j^2 < \infty$,
       \begin{equation}
              \frac{1}{(1+\varepsilon)^{j-1}}\frac{1}{E(t_{i_{j}})} (i_{j+1} - i_j) = \Omega\left( \frac{1}{j^2} \right).
       \end{equation}
       From this the desired estimate~\eqref{E:estimate} is immediate.
\end{proof}

\section{Illustrating examples}
In this section, we show some examples of Section~\ref{sec3} and derive essential convergence rates.
Wibisono et al.~\cite{Wibisono} showed that for a continuously differentiable convex function $f$ and for any differentiable monotonically increasing function $\eta : \RR \to \RR$, the solution $(x,z)$ of 
\begin{equation}
       \left\{ 
       \begin{aligned}
              \dot{x} &= \frac{\gamma(t)}{\re^{\eta(t)}} (z - x),\\
              \dot{z} &= -\gamma(t) \nabla f(x),
       \end{aligned}
       \right. \label{Wibisono}
\end{equation}
where $\gamma(t) := (\dd / \dd t)\re^{\eta(t)}$, satisfies 
\begin{equation}
       f(x(t)) - f^\star = \Order{\re^{-\eta(t)}}. \label{Wibirate}
\end{equation}
Here, the case of $\re^{\eta(t)} = t^2 / 4$ corresponds to \eqref{suODE}.
Since $\eta(t)$ is arbitrary, the rate can be arbitrarily fast.

ODE~\eqref{Wibisono} can be understood in the following way.
First we notice that the solution $(X,Z)$ of
\begin{equation}
       \left\{ 
       \begin{aligned}
              \dot{X} &= \frac{1}{\tau} (Z - X),\\
              \dot{Z} &= -\nabla f(X)
       \end{aligned}
       \right.
\end{equation}
satisfies 
\begin{equation}
       f(X(\tau)) - f^\star = \Order{1/\tau}
\end{equation}
Then the time rescaling by $\tau = \re^{\eta(t)}$ yields $(x,z)= (X(\re^{\eta(t)}), Z(\re^{\eta(t)}))$, which is the solution of ODE~\eqref{Wibisono} satisfying the convergence rate~\eqref{Wibirate}.

We now attempt to apply Theorem~\ref{thm4} to ODE~\eqref{Wibisono} and consider the essential convergence rate of this dynamical system.
For simplicity, we consider the case of $\re^{\eta(t)} = t^p\,(p>0)$, by which ODE~\eqref{Wibisono} reads
\begin{equation}
       \left\{ 
       \begin{aligned}
              \dot{x} &= \frac{p}{t} (z - x),\\
              \dot{z} &= -pt^{p-1} \nabla f(x).
       \end{aligned}
       \right.       \label{testode}
\end{equation}
In this case the objective function decreases at rate $\Order{1/t^p}$ by the solution $x$ of the above system.
By linearizing $\nabla f(x)$, which we denote by $a x$, 
the eigenvalues of the Jacobian of \eqref{testode}'s right-hand side are asymptotically
\begin{equation}
       \lambda \approx \pm \sqrt{a} p t^{\frac{p}{2}-1} \ri,
\end{equation}
where $\ri$ is the imaginary unit.
When discretizing this system, we have to choose the step size $h_k$ so that $h_k \lambda$ lies in the bounded stability domain. Thus $\abs{h_k \lambda} < \text{(const.)}$ holds, which implies $h_k = \Order{t^{1-\frac{p}{2}}}$.
When $p = 2$, the ODE is a proper representative, where $h_k$ can be taken to a constant size $h$.
After $k$ steps the elapsed time of the system~\eqref{testode} is $t = kh$ and therefore the objective function can decrease at rate $\Order{1/k^2}$.
If $p > 2$, however, $h_k$ must be taken gradually smaller and thus $k$ steps do not simply mean that the integration amounts to some time proportional to $k$.
Hence the convergence rate $\Order{1/k^p}$ cannot be achieved despite the rate $\Order{1/t^p}$ in continuous time.
By Theorem~\ref{thm4}, the discrete-time rate cannot exceed the essential convergence rate $\Order{1/k^2}$.
Note that even if we happen to once choose a ``slow'' scale $p=1$, the essential rate $\Order{1/k^2}$ can be recovered by taking $h_k = \Theta(k)$. In this case, Theorem~\ref{thm4} states that the recovery cannot exceed the essential rate.

Next, we show a case where the essential convergence rate cannot be recovered from some time scales.
By fixing $f$, we can discuss the convergence rate in more detail.
Setting $f(x) = x^4 / 4$ and let us derive the proper representative of the gradient flow (whose rate is $\Order{1/t}$ for differentiable convex functions):
\begin{equation}
       \dot{x} = -\nabla f(x) = -x^3, \quad x(0) = 1. \label{x3}
\end{equation}
The solution is written as
\begin{equation}
       x(t) = \frac{1}{\sqrt{2t + 1}}
\end{equation}
and thus the convergence rate is $\Theta(1/t^2)$.
However, the Jacobian of \eqref{x3}'s right-hand side is
\begin{equation}
       \pdv{x}(-x^3) = \frac{-3}{2t + 1}, \label{nlin}
\end{equation}
which implies ODE~\eqref{x3} is not a proper representative.
The proper representative is as follows:
\begin{equation}
       \dot{x} = -\re^t x^3,
\end{equation}
since the Jacobian of \eqref{x3}'s right-hand side is
\begin{equation}
       \pdv{x}(-\re^t x^3) = \frac{-3 \re^t}{2(\re^t - 1) + 1} = \Theta(1),
\end{equation}
and the essential convergence rate is $\Theta(1/\re^{2t})$.
Here we can see that it is impossible to recover the essential convergence rate by discretizing ODE~\eqref{x3} as long as step sizes respect the stability domain; \eqref{nlin} implies that the increase in the step sizes without violating the stability domain is up to a linear scale, though the exponential growth is necessary to restore the essential convergence rate. By actual computation, we observe that in fact such aggressive growth is allowed numerically. This phenomenon is because \eqref{x3} is a purely nonlinear ODE, while the step-size restriction is based on linear stability analysis. Note that this failure does not contradict to Theorem~\ref{thm4}; it does not claim $\alpha(t_k) - \alpha(t_0) = \Theta(k)$ but $\Order{k}$.


Theorem~\ref{thm4} cast a strong restriction on accelerations by time rescaling, but there may remain a loophole. We have considered fixed numerical schemes in this letter, but if the scheme changes during time evolution, especially if the stability domain expands, it is possible to exceed the limit of this theorem at least formally. This is an interesting topic, and worth further investigation.

\end{document}